\documentclass{amsart}

\usepackage[dvipsnames]{xcolor}
\usepackage{tikz}
\usepackage{amsmath,bm,bbm,amsthm, amssymb}
\usepackage{mathtools}
\usepackage{fullpage}
\usepackage{color}
\usepackage{dsfont}
\usepackage{animate}
\usepackage{etoolbox}
\usepackage{comment}
\usepackage{pgf}
\usepackage{cite}

\theoremstyle{plain}
\newtheorem{theorem}{Theorem}

\newtheorem{corollary}[theorem]{Corollary}
\newtheorem{lemma}[theorem]{Lemma}

\theoremstyle{definition}
\newtheorem{definition}[theorem]{Definition}

\theoremstyle{remark}
\newtheorem{remark}[theorem]{Remark}

\def\E{\mathbb E}
\def\P{\mathbb P}
\def\R{\mathbb R}
\def\one{\mathds1}
\def\Cov{\mathsf{Cov}}
\def\d{\mathrm d}

\begin{document} 
\author{Patrick Lopatto}
\address{Department of Statistics and Operations Research, University of North Carolina at Chapel Hill, United States}
\email{lopatto@unc.edu}
\author{Moritz Otto}
\address{Mathematical Institute, Leiden University, The Netherlands}
\email{m.f.p.otto@math.leidenuniv.nl}


\title{Maximum gap in complex Ginibre matrices} 
\date{\today}
\begin{abstract} 
We determine the asymptotic size of the largest gap between bulk eigenvalues in complex Ginibre matrices. 
\end{abstract}
\maketitle

\section{Introduction}\label{sec:intro}

Random matrices have been studied for decades due to their rich mathematical structure and applications in quantum physics, statistics, and engineering. The behavior of their eigenvalues, and especially the spacings between eigenvalues, is a central topic. Due to recent advances, there is now a mature theory that identifies the distribution of a single eigenvalue gap for a broad class of random matrix models \cite{erdos2019matrix,erdHos2017dynamical}. 
However, much is less known about the largest and smallest eigenvalue gaps, in part because studying these observables requires understanding the entire collection of gaps simultaneously. Besides being fundamental examples in the extreme value theory of strongly correlated systems, extremal eigenvalue gaps have attracted attention due to a conjectural link to the Riemann zeta function. 
Numerical evidence suggests that the largest and smallest spacings of zeros of the zeta function and the corresponding extremal eigenvalue gaps are identically distributed (after appropriate rescaling) \cite{odlyzko1987distribution,ben2013extreme}. This connection can be seen as an extremal version of Montgomery's pair correlation conjecture \cite{montgomery1973pair}. 

In the literature on extremal eigenvalue gaps, the majority of works have concentrated on models with one-dimensional spectrum. The classical Gaussian random matrix ensembles, which include the Gaussian Orthogonal Ensemble (GOE), Gaussian Unitary Ensemble (GUE), and Gaussian Symplectic Ensemble (GSE), are paradigmatic 
examples. 
The correlation functions for these ensembles admit explicit formulas, allowing their eigenvalue behavior to be studied through direct computations. As a result, the distributions of the smallest gaps are known for each of the classical Gaussian ensembles \cite{vinson2001closest,ben2013extreme,feng2019small,feng2024small}. The problem of determining the distributions of the largest gaps seems more challenging. Currently, these distributions are known only for the GUE; they were determined in \cite{feng2018large}, following earlier work in \cite{ben2013extreme} that obtained the asymptotic size of the largest gaps. Further, it was shown in \cite{bourgade2021extreme} that these results on the largest and smallest gaps are universal in the sense that they hold for all Wigner matrices, assuming the entry distributions posses a density satisfying a weak smoothness condition.\footnote{Roughly speaking, Wigner matrices are symmetric random matrices whose upper-triangular entries are independent with mean zero and variance one.}
In the case of the largest gap, universality without the smoothness condition was shown in \cite{landon2020comparison}; this result encompasses, for example, Bernoulli random matrices. 
Very recently, universality for the smallest gap was extended to a large class of matrices containing non-smooth distributions, include those with sufficiently many atoms \cite{zhang2025quantitative}.

The non-symmetric analogues of the classical Gaussian ensembles are the Ginibre ensembles \cite{ginibre1965statistical,byun2024progress}, including the orthogonal Ginibre ensemble (GinOE) and unitary Ginibre ensemble (GinUE). These have two-dimensional spectra, each asymptotically supported on the unit disk in the complex plane.
The smallest gaps distribution of the GinUE was determined in \cite{shi2012smallest}, and the analogous result for the GinOE was established in \cite{lopatto2024smallest}. However, even the asymptotic size of the largest gaps for such ensembles has not been rigorously understood, much less their asymptotic distributions. In this paper, we determine the asymptotic order of the largest gap between bulk eigenvalues of the GinUE, including the leading constant factor.  To the best of our knowledge, this is the first such result for any random matrix ensemble with independent, identically distributed entries. 

\subsection{Proof ideas} We consider an auxiliary point process, denoted $\Xi_{n}$ below, that consists of the set of $n$ GinUE eigenvalues thinned to retain only those points whose near-neighbor spacing is greater than a certain distance $r_{n}$, defined precisely below.    
Our main result, Theorem~\ref{th}, states that on any set $B\subset \mathbb{C}$ such that $\sup_{z \in B} |z|<1$, the Kantorovich--Rubinstein distance between $\Xi_{n}$ and a suitably chosen Poisson process $\zeta_c$ tends to zero as $n \rightarrow \infty$. Routine arguments then allow one to deduce the leading-order asymptotic for the nearest-neighbor spacing in $B$, stated in Corollary~\ref{cor:maxdistance}. 

Our approach is similar to \cite{O24}, which determined the leading-order asymptotic for the largest gap for the limiting process of the GinUE eigenvalues. However, \cite{O24} focused on deriving a general Poisson process approximation result for stabilizing functionals of determinantal point processes. The error terms in this approximation are relatively straightforward to bound for the limiting point process by exploiting its stationarity. Our emphasis is different. We rely on the simpler Poisson approximation result \cite[Theorem 3.1]{BSY21}, but face new technical challenges because the eigenvalue process of the GinUE is not stationary. 
In particular, we would like to apply \cite[Theorem 3.1]{BSY21} to $\Xi_{n}$, but that result requires to control the symmetric difference of $\Xi_{n}$ and its reduced Palm version at every $z \in B$, which seems difficult to accomplish. However, due to the stabilization property of the eigenvalue process, we can construct a coupling (with a certain monotonicity property) of these processes after restricting them to a set that omits a small ball around $z$. By adapting the proof of  \cite[Theorem 3.1]{BSY21}, we show that the small set where we cannot construct a coupling can be neglected, then bound the error terms arising on the complementary large set (without relying on stationarity, in contrast to \cite{O24}). These bounds rely on large deviations estimates for the probability that a given region contains at most two GinUE eigenvalues.

\subsection{Related works} The circular ensembles, which have one-dimensional spectra, are another class of classical matrix ensembles. The distribution of the smallest gaps for the CUE was derived in \cite{ben2013extreme}, along with the order of the largest gaps. The distributions of the largest gaps were found in \cite{feng2018large}. Small gaps of the circular $\beta$-ensembles for integer $\beta$, a more general model which specializes to the CUE in the case $\beta=2$, were studied in \cite{feng2021small}. The largest and smallest gaps for a multi-matrix ensemble that generalizes the GUE were studied in \cite{figalli2016universality}. For the spherical ensemble, a determinantal process of points on the sphere that can be realized as the stereographic projection of the eigenvalues of a certain random matrix, the asymptotic area of the largest empty cap was derived in \cite{alishahi2015spherical}.

In models where obtaining the full distribution of the smallest gap is challenging, one can instead ask for high-probability lower bounds. These have been derived for symmetric random matrices in 
\cite{nguyen2017random,lopatto2021tail} and matrices of independent entries in \cite{Ge2017EigenvalSpacing,Luh2021Eigenvectors}.
For the Coulomb gas model, which generalizes the GinUE, various separation estimates for the particles have been shown in \cite{ameur2018repulsion,ameur2023planar,thoma2022overcrowding, Ameur2024HeleShaw}, and recently the convergence of the smallest gaps process to a Poisson process was established in \cite{charlier2025smallest}, generalizing the GinUE result of \cite{shi2012smallest}. Largest gap behavior is closely related to existence of large zero-free regions, which have been studied in \cite{Charlier2024Annuli, adhikari2017hole, adhikari2018hole,charlier2023hole}. 

The limiting distribution of the $r$-th largest nearest-neighbor spacing between independent points in $\mathbb{R}^p$, for fixed $r, p \ge 1$, was determined in \cite{henze1982limit,henze1983asymptotischer}. Poisson convergence for these spacings was studied in \cite{cho,BSY21}, and the analogue for independent points in hyperbolic space was established in  \cite{otto2023large}.

\section{Model and main results} \label{sec:model}

We  work on the complex plane $\mathbb C$ equipped with its Borel $\sigma$-field $\mathcal B$. We denote by $\mathbf N$  the space of all $\sigma$-finite counting measures on $\mathbb{C}$ and by $\widehat{\mathbf{N}}$ the space of all finite counting measures on $\mathbb{C}$.
We equip $\mathbf N$ and $\widehat{\mathbf{N}}$ with their corresponding $\sigma$-fields $\mathcal N$ and $\widehat{\mathcal N}$, which are induced by the maps $\omega \mapsto \omega(B)$ for all $B \in \mathcal B$. A {\em point process} is a random
element $\nu$ of $\mathbf N$, defined over some fixed
probability space $(\Omega,\mathcal A,\P)$.
The {\em intensity measure} of $\nu$ is the
measure $\E[\nu]$ defined by $\E[\nu](B)=\E[\nu(B)]$ for all $B\in\mathcal B$. 
For all $z \in \mathbb{C}$ and $r>0$,  let $B_r(z)$ be the closed Euclidean ball with radius $r$ around $z$. For all $B \in \mathcal B$, we write $|B|$ for the Lebesgue measure of $B$.

Recall that the complex Ginibre matrix of dimension $n$ is an $n\times n$ matrix whose entries are independent, identically distributed random variables, with common distribution equal to the standard complex Gaussian distribution; this means that their real and imaginary parts are each Gaussian with mean zero and variance $1/2$. 
We let $\xi_n$ denote the point process formed by the $n$ eigenvalues of the $n$-dimensional Ginibre matrix.
We remark that with the scaling we have chosen, the largest modulus of an element of $\xi_n$ is roughly $\sqrt{n}$.

Recall that a point process $\nu$ is called a {\em determinantal point process with correlation kernel $K$} on $\mathbb{C}$ if for every $m\in \mathbb N$ and pairwise disjoint $A_1,\dots,A_m \in \mathcal B$, 
\begin{align*}
	\E [\nu(A_1)\cdots \nu(A_m)]=\int_{A_1 \times \cdots \times A_m} \det (K(z_i,z_j))_{i,j=1}^m \d z_1 \, \dots\, \d z_m, 
\end{align*} 
where $\d z_i$ denotes integration with respect to the Lebesgue measure on $\mathbb{C}$, $(K(z_i,z_j))_{i,j=1}^m$ is the $m\times m$-matrix with entry $K(z_i,z_j)$ at position $(i,j)$, and $\det M$ is the determinant of the complex-valued $m\times m$-matrix $M$. 
Then $\nu$ has correlation functions of all orders, and the $m$-th order correlation function $\rho_\nu^{(m)}$ is given by
\begin{equation*}
	\rho^{(m)}_\nu(z_1,\dots,x_m)=	\det (K(z_i,z_j))_{i,j=1}^m,\quad z_1,\dots,z_m \in \mathbb{C},\,\quad m \in \mathbb{N}.
\end{equation*}
It is well known (see, e.g.,  \cite[Theorem 4.3.10]{BKPV09}) that $\xi_n$ is a determinantal point process with kernel
\begin{equation}\label{e:ginkernel}
K_n(z,w)=\sum_{k=1}^n \varphi_k(z) \overline{\varphi_k(w)},\quad \text{where }\quad \varphi_k(z):=\frac{1}{\sqrt{\pi (k-1)!}} z^{k-1} e^{- \frac{1}{2} |z|^2}.
\end{equation}

In this article, we study the largest gap in the process $\xi_n$. To do so, we let $\kappa>1$ be a parameter and consider the point process-valued function
\begin{align}
	\Xi_{n}[\omega]:=\sum_{z \in \omega \cap B_{\sqrt n}(0)} \one\{(\omega \setminus \{z\}) (B_{r_{n}(z)}(z))=0\} \,\delta_{z/\sqrt n}, \quad \omega \in \mathbf N, \label{def:Xin}
\end{align}
where the sequence $(r_{n}(z))_{n \in \mathbb N}$ is given by
\begin{align}
r_{n}(z):=\inf\{r>0:\, n \P(\xi_n^{z!}(B_{r}(z))=0) \rho_n(z)\le \kappa \},\quad z \in \mathbb C.\label{def:rcn}
\end{align}
Here $\rho_n(z):=K_n(z,z)$ is the one-point correlation function for $\xi_n$, and $\xi_n^{z!}$ denotes the reduced Palm process of $\xi_n$ at $z$ (see Section \ref{sec:prelim} for its definition).
We define  $\Xi_{n}=\Xi_{n}[\xi_n]$ and let $L_n(A)= \E [\Xi_n(A)]$ denote the intensity measure of $\Xi_n$, where $A \subset \mathbb{C}$ is assumed to be measurable. While $\Xi_{n}$, $r_n$, and $L_n$ depend on $\kappa$, we  omit this dependence in the notation.

We study the Kantorovich--Rubinstein (KR) distance between the distributions of $\Xi_n$ and a certain finite Poisson process. We recall the definition of the KR distance from \cite{DST16}. For finite point processes $\zeta$ and $\nu$ on $\mathbb{C}$, we have 
\begin{align}
	{d_{\mathrm{KR}}}(\zeta, \nu) := \sup_{h \in \text{Lip}}\big|\mathbb{E} [h(\zeta)]-\mathbb{E} [h(\nu)]\big|,\label{dkr}
\end{align}
where $\text{Lip}$ is the class of all measurable 1-Lipschitz functions $h\colon \widehat{\mathbf{N}} \to \mathbb{R}$ with respect to the total variation between measures $\omega_1, \omega_2$ on $\mathbb{C}$, which is given by
\begin{align*}
	d_{\text{TV}}(\omega_1,\omega_2):=\sup |\omega_1(A)-\omega_2(A)|,
\end{align*}
where the supremum is taken over all $A \in \mathcal B$ with $\omega_1(A), \omega_2(A)<\infty$. 

We now state our two results, which are proved in Section~\ref{sec:pr1}. The first shows that $\Xi_{n}$ may be approximated by a Poisson process at a resolution small enough to capture the behavior of the largest eigenvalue gap. 
\begin{theorem} \label{th}
	Fix  $s \in (0,1)$, $\epsilon >0$, and $\kappa>1$, and recall that $\xi_n$ denotes the point process of eigenvalues for the $n$-dimensional complex Ginibre matrix, and that $\Xi_{n}:=\Xi_{n}[\xi_n]$ is the process defined at \eqref{def:Xin}.   
There exists $C(s, \epsilon, \kappa)>1$ such that for all $B \subset B_1(0)$ satisfying $\sup_{z \in B} | z |  <s$, we have  for all  $n \in \mathbb{N}$ that 
	\begin{align}\label{e:KRthm}
		&{d_{\mathrm{KR}}}(\Xi_{n} \cap B,\zeta_\kappa \cap B) \le C n^{\epsilon-1/16},
	\end{align}
where $\zeta_\kappa$ is a Poisson process with intensity measure $L_n$.
Further,  we have for all $\kappa>1 $ that as $n \to \infty$,
\begin{equation}\label{e:rcnlimit}
\sup_{z \in \sqrt n B} \Big|\frac{r_{n}(z)^4}{4 \log n} -1\Big| \to 0.
\end{equation}
\end{theorem}

Our second result is a consequence of the previous theorem. It obtains the asymptotic scaling of the largest distance from a particle  in $ \xi_n$ to its nearest neighbor. This scaling is different than the one for independent points, which is order $\log n$ \cite{cho}. 

\begin{corollary}
	\label{cor:maxdistance}
	For all $B \subset B_1(0)$ such that $\sup_{ z \in B} | z | <1$, we have that as $n \to \infty$,
	$$
	\frac{1}{4 \log n} \max_{z \in \xi_n \cap \sqrt nB} \min_{w \in \xi_n \setminus \{z\}} | z-w|^4 \stackrel{\P}{\longrightarrow} 1.
	$$
\end{corollary}

\section{Preliminaries} \label{sec:prelim}
\subsection{Palm calculus}

We begin by introducing three point processes that will play fundamental roles in the proof of Theorem \ref{th}.
\begin{definition} \label{d:palm}
	Let $z \in \mathbb C$. The following equalities are assumed to hold for all measurable $f\colon \mathbb{C} \times \mathbf{N} \to [0,\infty)$.
	\begin{itemize}
		\item[(a)]  $\xi_n^z$ is a Palm process of $\xi_n$ at $z$ if it satisfies
		\begin{align}
			\E \left[\int_{ B_{\sqrt n}(0)} f(z,\xi_n)\xi_n(\d z )\right]&=\int_{ B_{\sqrt n}(0)} \E \big[f(z,\xi_n^{z}) \big]\rho_n(z) \, \d z. \label{defPalm}
		\end{align}
	\item[(b)]  $\Xi_n^z$ is a Palm process of $\Xi_n$ at $z$ if it satisfies
	\begin{align}
		\E \left[\int_{ B_{1}(0)} f(z,\Xi_n)\Xi_n(\d z )\right]&=\int_{ B_{1}(0)} \E\big [f(z,\Xi_n^{z}) \big]L_n(\d z). \label{defPalm0}
	\end{align}
	\item[(c)]  $\xi_n^{z,\varnothing}$ is a Palm process of $\xi_n$ with respect to $\Xi_n$ at $z$ if it satisfies\footnote{The superscripted $\varnothing$ serves as a reminder of the indicator function for  the event $(\xi_n \setminus \{z\}) (B_{r_n}(z)) = \varnothing$ present in the definition of $\Xi_n$.}
\begin{align}
\E \left[\int_{B_{1}(0)} f(z,\xi_n)\Xi_n(\d z )\right]&=\int_{ B_{1}(0)} \E \big[f(z,\xi_n^{z,\varnothing}) \big]L_n(\d z). \label{defPalm1}
\end{align}
\end{itemize}
The processes $\xi_n$, $\xi_n^{z!}:=\xi_n^z\setminus \{z\}$, $\Xi_n^{z!}:=\Xi_n^z \setminus \{z\}$ and $\xi_n^{z!,\varnothing}:=\xi_n^{z,\varnothing}\setminus \{z\}$ are called {\em reduced Palm processes}.
\end{definition}

It follows from standard results that for all $z  \in \mathbb C$ and  $n \in \mathbb{N}$, there exist  point processes $\xi_n^{z}$, $\Xi_n^{z}$, and $\xi_n^{z, \varnothing}$ satisfying parts (a), (b), and (c) of Definition~\ref{d:palm}, respectively \cite[Chapter 6]{K17}. Further, \cite[Lemma 6.2(ii)]{K17} implies $z \in \xi_n^{z}$ almost surely, and parallel claims hold for $\Xi_n^{z}$ and $\xi_n^{z, \varnothing}$.
\begin{remark}\label{r:XiPalm}
Comparing 
\eqref{defPalm0} and \eqref{defPalm1}, we find that that $\Xi_n[\xi_n^{z,\varnothing}]$ is a Palm process of $\Xi_n$ at $z$. Indeed, if we choose $\tilde f(x,\omega):=f(x,\Xi_n[\omega])$ in \eqref{defPalm1}, then its left-hand side coincides with the left-hand side of \eqref{defPalm0}, and comparing their right-hand sides gives the claim.
\end{remark}

 We write $P_n$, $P_n^{z}$ and $P_n^{z, \varnothing}$ for the distributions of $\xi_n$, $\xi_n^{z}$ and $\xi_n^{z,\varnothing}$, respectively, and $P_n^{z!}$ and $P_n^{z!, \varnothing}$ for the distributions of the associated reduced Palm processes. For brevity, we write $B_{r_n}(z)$ instead of $B_{r_n(z)}(z)$.
 
 For every $x\in \mathbb{C}$,  $\xi_n^{x!}$ is a  determinantal process with correlation kernel $K_n^x$ given by
\begin{equation}
	K_n^x(z,w)=K_n(z,w)-\frac{K_n(z,x)K_n(x,w)}{K_n(x,x)},\quad z,w \in \mathbb C \label{KPalm}
\end{equation} 
(see \cite[Theorem 1.7]{ST03}).

\begin{lemma}\label{l:requality}
For every $s \in (0,1)$ and $\kappa >1$, there exists $C(s, \kappa)>1$ such that for all $z \in B_{s \sqrt{n}}(0)$ and $n \ge C$, 
\begin{equation}
n \P(\xi_n^{z!}(B_{r_n}(z))=0) \rho_n(z) = \kappa.
\end{equation}
\end{lemma}
\begin{proof}
It is straightforward to see that $  \rho_n( z_n) \uparrow \frac 1 \pi$ uniformly as $n \to \infty$ for all sequences $(z_n)_{n=1}^\infty$ such that $z_n \in \sqrt{n}B$. Then there exists $C(s) > 1$ such that for all $z \in B_{s \sqrt{n}}(0)$ and $n \ge C$, 
\begin{equation}\label{e:completeproof}
	\frac 1{2\pi  } \le \rho_n(z) \le \frac 1{\pi }. 
\end{equation}
Using that $\xi^{z!}_n$ is determinantal (see \eqref{KPalm}), we have for every $r>0$ that 
\begin{equation}\label{e:integralrep}
\mathbb P(\xi^{z!}_n(B_r(z) )=0)
=
\frac{1}{n!}
\int_{B_r(z)^c}\dots \int_{B_r(z)^c} \det (K^z_n (z_i,z_j))_{i,j=1}^m  \, dz_1\dots dz_m.
\end{equation}
Together with \eqref{KPalm}, this representation shows that the function $g(r) = \P(\xi^{z!}_n(B_r(z) )=0)$ is a continuous function of $r$ for $r \in [0, \infty)$. Further, it is clear from \eqref{e:integralrep} that $g(0) = 1$ and $ \lim_{r \rightarrow \infty} g(r) = 0$. The claim follows from combining these limits with \eqref{e:completeproof}.
\end{proof}

\begin{lemma}\label{l:int}
For every $s \in (0,1)$ and $\kappa >1$, there exists $C(s, \kappa)>1$ such that for all $A \subset B_{s}(0)$ and $n \ge C$,  it holds that $L_n(A)=c|A|$.
\end{lemma}
\begin{proof}
	By the definition of $\Xi_n$ and \eqref{defPalm} with $f(z, \omega) = \one\{\omega(B_{r_n}(z ))=1\}$ we obtain
	\begin{align*}
	L_n(A) &= 	\E \Big[\sum_{z \in \xi_n} \one\{z/\sqrt n \in A\} \one \{\xi_n(B_{r_n}(z ))=1\}\Big]\\
	&= \E \Big[\int_{\sqrt n A}  \one \{\xi_n(B_{r_n}(z ))=1\} \, \xi_n(\d z)\Big]= \int_{\sqrt nA} \P(\xi_n^z(B_{r_n}(z))=1) \rho_n(z) \, \d z.
	\end{align*}
	Now the assertion follows from \eqref{def:rcn} and Lemma~\ref{l:requality}. 
\end{proof}

\subsection{Stochastic domination}
By \cite[Theorem 3]{G10} (see also \cite{MR21}), the process $\xi_n^{x!}$ is stochastically dominated by $\xi_n$ (denoted by $P_n^{x!}\le P_n$), which means that
\begin{align}
	\mathbb{E} [F(\xi_n^{x!})] \le \mathbb{E} [F(\xi_n)] \label{Palmdom}
\end{align}
for every measurable $F\colon \mathbf{N} \to \mathbb{R}$ that is bounded and increasing; the latter condition means that $F(\omega_1) \le F(\omega_2)$ if $\omega_1 \subset \omega_2$.

An important property of determinantal point processes is that they have negative associations (see \cite[Theorem 3.7]{L14}). We say that a point process $\nu$ {has negative associations} if 
\begin{align}
{	\mathbb E[F(\nu)G(\nu)] \le  \mathbb E[F(\nu)] \E[G(\nu)],} \label{defNA}
\end{align}
for any  real, bounded, and increasing functions $F,G\colon \mathbf N \to \mathbb R$ that are measurable with respect to complementary subsets (see \cite{LS19}).

We also require the following lemma concerning stochastic domination.

\begin{lemma}\label{l:stocdom}
For every $s \in (0,1)$ and $\kappa >1$, there exists $C(s, \kappa)>1$ such that for all  $n \ge C$,
we have for $L_n$-almost all $z \in B_s(0)$ that
\begin{align}\label{xiPalmdomincr}
P_n^{\sqrt{n} z, \varnothing}|_{B_{r_{n}}(\sqrt nz)^c}\ge P_n^{\sqrt n z}|_{B_{r_{n}}(\sqrt nz)^c}.
\end{align}
\end{lemma}
\begin{proof}
Let $F\colon \mathbf{N} \to \mathbb{R}$ be bounded and increasing. For all Borel sets $A$ with $\sup_{z \in A} |z| <s$ we have
\begin{align}
 \int_{A} \E [F(\xi_n^{\sqrt{n} z, \varnothing} \cap B_{r_{n}}(\sqrt nz)^c)] L_n(\d z)
	&= \E \left[\int_{\sqrt n A}   F(\xi_n\cap B_{r_{n}}( z)^c) \one\{(\xi_n-\delta_z)(B_{r_{n}}(z))=0 \} \xi_n(\d z) \right]\nonumber\\
	& =  \int_{\sqrt n A}\E \big[F(\xi_n^z\cap B_{r_{n}}(z)^c) \one\{\xi_n^{z!}(B_{r_{n}}(z))=0 \}\big]  \rho_n(z)\, \d z\nonumber\\
	& =  n \int_{A}\E\big [F(\xi_n^{\sqrt n z}\cap B_{r_{n}}(\sqrt nz)^c) \one\{\xi_n^{\sqrt n z!}(B_{r_{n}}(\sqrt n z))=0 \} \big]  \rho_n(\sqrt n z) \, \d z\nonumber\\
	&\ge n \int_{A}\E \big[F(\xi_n^{\sqrt n z}\cap B_{r_{n}}(\sqrt nz)^c) \big] \P(\xi_n^{\sqrt n z!}(B_{r_{n}}(\sqrt n z))=0) \rho_n(\sqrt n z)\, \d z\nonumber \\
	&=\int_{A}\E \big[F(\xi_n^{\sqrt n z}\cap B_{r_{n}}(\sqrt nz)^c)\big] L_n(\d z),\nonumber
\end{align}
where we have used \eqref{defPalm0}, Lemma~\ref{l:requality}, and a change of variables in the first line, \eqref{defPalm} in the second line, a change of variables in third line, and \eqref{defNA} together with the fact that $\xi_n^{\sqrt nz!}$ is a determinantal process (and thus  has negative associations) for the inequality in the fourth line. 
The fifth line follows from \eqref{def:rcn} and  Lemma~\ref{l:requality}. This gives the claim.
\end{proof}

\subsection{Coupling of $P_n$ and $P_n^{z!}$} \label{sec:coup}
For a fixed $z \in \mathbb{C}$, we now specify a coupling of $P_n$ and $P_n^{z!}$. 
This will help us bound the KR distance of $\Xi_n$ and a certain Poisson process in Section \ref{sec:pr1}.
By \cite[Theorem 1]{MR21},  there exists a coupling $(\xi_n, \xi_n^{z!} )$ of $\xi_n$ and $\xi_n^{z!}$  such that $\xi_n^{z!} \subset \xi_n$ and $|\xi_n^{z!}  \setminus \xi_n  |\le 1$ almost surely.  The hypotheses of \cite[Theorem 1]{MR21} are straightforward to check using the representation \eqref{e:ginkernel}. 
We will always use this coupling when discussing the relationship between $\xi_n$ and $\xi_n^{z!}$ in what follows. Moreover, \cite[Theorem 1]{MR21} also explicitly gives the conditional distribution of $\xi_n$ conditional on $\xi_n^{ z!} $. 

Similarly, for any distinct $w,z \in \mathbb{C}$, there exists a coupling of $\xi_n^{z!}$  and $\xi_n^{z!, w!}  := (\xi_n^{z!})^{w!}$. The hypotheses of \cite[Theorem 1]{MR21} may be checked in this case using \eqref{KPalm}.

Additionally, under the conditions of Lemma~\ref{l:stocdom}, we specify a coupling of $P_n^{\sqrt n z!}$ and  $P_n^{\sqrt{n} z!,\varnothing}$ on $B_{r_n}(\sqrt n z)^c$. We use in coupling in 
Section~\ref{sec:pr1} to construct a coupling of $\Xi_n$ and $\Xi_n^{z!}$ outside of a ball around $z$. 
We do not use \cite[Theorem 1]{MR21} for this purpose, since its hypotheses are harder to check in this case, and we instead settle for a weaker statement.  
Fix $z \in B_1(0)$ and $\xi_n^{\sqrt{n} z!} \sim P_n^{\sqrt{n} z!}$, and observe \eqref{xiPalmdomincr} gives that $P_n^{\sqrt n z!} \le P_n^{\sqrt{n} z!,\varnothing}$ on $B_{r_n}(\sqrt n z)^c$ (for sufficiently large $n$).
By Strassen's theorem (see \cite{lindvall}), there exists a measure $\rho \sim P_n^{\sqrt n z!} \otimes P_n^{\sqrt{n} z!,\varnothing}$ supported on $\{\omega_1 \otimes \omega_2 \in \mathbf N \times \mathbf N: \omega_1 \subset \omega_2\}$. By \cite[Theorem 3.4(i)]{K03}, $\rho$ can be disintegrated and, hence, there exists a point process $\xi_n^{z!,\varnothing}  \sim P_n^{\sqrt{n} z!,\varnothing}$ such that, almost surely,
\begin{equation}\label{e:previnc}
\xi_n^{\sqrt n z!}|_{B_{r_{n}}(\sqrt nz)^c} \subset \xi_n^{\sqrt{n} z!,\varnothing}|_{B_{r_{n}}(\sqrt nz)^c} .
\end{equation}

Combining the above couplings gives a joint coupling of $\xi_n|_{B_{r_{n}}(\sqrt nz)^c}$, $\xi_n^{\sqrt n z!}|_{B_{r_{n}}(\sqrt nz)^c}$, and  $\xi_n^{\sqrt{n} z!,\varnothing}|_{B_{r_{n}}(\sqrt nz)^c}$ such that 
$\xi_n^{\sqrt n z!}|_{B_{r_{n}}(\sqrt nz)^c} \subset \xi_n|_{B_{r_{n}}(\sqrt nz)^c}$ and \eqref{e:previnc} hold.

\subsection{Fast decay of correlation}
We begin with a preliminary bound on $K_n(z,w)$. 
\begin{lemma}\label{l:Kbound}
For all $B \subset B_1(0)$ such that $\sup_{z \in B} | z | <1$, there exists a constant $c(B)>0$ such that 
\[
\sup_{z,w \in B}
\big| K_n(\sqrt{n} z, \sqrt{n} w ) \big|
\le \frac{1}{\pi} \exp\left( - \frac{n}{2} |z-w|^2 \right) + c^{-1} e^{-cn}.
\]
\end{lemma}
\begin{proof}
By the definition of $K_n$ and \cite[Lemma 4.1]{goel2024central},
\begin{align*}
K_n(\sqrt{z},\sqrt{w})
&=
\frac{1}{\pi} e^{- n(|z|^2 + |w|^2)/2}\sum_{k=0}^{n-1} \frac{ n^{k-1} (z\bar w)^{k-1}}{k!}\\
& =\frac{1}{\pi} e^{- n(|z|^2 + |w|^2)/2} \left( e^{n z\bar w}   - \frac{ e^{ n z\bar w} }{\sqrt{2 \pi n}} \frac{(z \bar w e^{1-z\bar w })^n }{1-z\bar w } \big( 1 + R_n(z,w ) \big) \right),
\end{align*}
where $R_n$ satisfies 
\[
\sup_{z,w \in B} \big|  R_n(z,w) \big| \le C n^{-1}
\]
for some constant $C>1$. 
We note that 
\[
|e^{- n(|z|^2 + |w|^2)/2}  e^{n z\bar w}| = e^{- n(|z|^2 + |w|^2)/2}  e^{n \Re( z\bar w)}  = e^{ - n |z -w|^2/2}.
\]
Further, for the second-order term,
\[
|e^{- n(|z|^2 + |w|^2)/2} (z \bar w)^n e^n | \le e^{- n(|z|^2 -  \ln |z|^2 - 1 + |w|^2 - \ln |w|^2  -1  )/2  }.
\]
The conclusion now follows from the elementary fact that there exists $c(B) > 0$ such that 
\[
x - 1  - \ln x > c 
\]
for $x$ such that $0 \le x \le \sup_{z \in B} |z|$, and the fact that $|1-z\bar w|$ is uniformly lower bounded for $z,w \in B$. 
\end{proof}

The following lemma concerns a mixing property of $\xi_n$. 
\begin{lemma}\label{l:lemma5}
Let $\rho^{(m)}$ be the $m$-th order correlation function of $\xi_n$. Fix $p, q \in \mathbb{N}$ and set $m= p+q$. 
For all $B \subset B_1(0)$ such that $\sup_{ z \in B_1(0)} | z | <1$, there exist constants $n_0(B), c(B)>0$ such that for all $x_1, \dots  , x_{p+q}$, 
\begin{align}
\big|\rho^{(p+q)}(x_1,\dots,x_{p+q})- \rho^{(p)}(x_1,\dots,x_p)\rho^{(q)}(x_{p+1},\dots,x_{p+q}) \big|\le  m^{1+m/2}  \left( e^{- ns^2} + c^{-1} e^{-cn} \right)  ,\label{eq:rhodec}
\end{align}
where  
\[ s=\inf_{i \in \{1,\dots,p\}, j \in \{p+1,\dots,p+q\}} |x_i-x_j|.
\]
\end{lemma}
\begin{proof}
Recall that for any determinantal point process with kernel $K(x,y)$, the restriction of the process to some set $S$ is also determinantal, with kernel $K(x,y) \one_S(x) \one_S(y)$. 
Then the conclusion is an immediate consequence of \cite[(3.4.5)]{AGZ10} and Lemma~\ref{l:Kbound}.
\end{proof}

\subsection{Infinite Ginibre process} 
Let $\xi$ be the (stationary) infinite Ginibre process on $\mathbb C$ with kernel \[K(z,w)=\frac {1}{\pi} 
e^{ - \frac{1}{2} |z|^2  - \frac{1}{2} |w|^2  + z \bar w }.\] 
Since $K_n(z,w) \to K(z,w)$ as $n \to \infty$, we have that $\xi_n \to \xi$ in distribution (see, e.g., \cite[Theorem 11.1.VII]{daley}). 

The following result is essentially a well-known theorem of Kostlan \cite[Theorem 1.1]{K92}, but we take the statement from \cite[Theorem 4.3.10]{BKPV09} to match our choice of scaling.  Let $( X_k)_{k=1}^{\infty}$ be a collection of independent random variables such that $X^2_k\sim \operatorname{Gamma}(k,1)$, where we use a shape--scale parameterization for the gamma distribution.

\begin{lemma}\label{l:gincounting}
The set of absolute values of the points of $\xi$ has the same distribution as $( X_k)_{k=1}^{\infty}$ (considered as a random set).
\end{lemma}

We note that by definition, $X^2_{k}$ has the same distribution as sum of the squares of $k$ independent  random variables, each having density $f(x) = e^{-x}$ for for $x\ge 0$. Note that  $|f(x)|\le 1$ for all $x \ge 0$, and the density $g_{k}$ of $X^2_{k}$ is the convolution of $f$ and the density $g_{k-1}$ of $X^2_{k-1}$. Then we have
\begin{equation}\label{e:gammadensitybound}
\| g_k \|_\infty = \| g_{k-1} \ast  f \|_\infty \le \| g_{k-1} \|_1 \| f \|_\infty \le 1,
\end{equation}
so the density of each $X_{k}$ is uniformly bounded by $1$.
 
\section{Auxiliary statements}

We begin by relating vacuum probabilities for the infinite Ginibre process to their analogues for the finite Ginibre process. 

\begin{lemma} \label{lem:gininf}
	For every $s \in (0,1)$, there exists $C(s)>1$ and  $\delta(s) <1$ with the following property.  For all $n  \in \mathbb N$ such that $n \ge C$ and any sequence of measurable sets $A_n \subset B_{s \sqrt n}$, we have 
	$$
	\mathbb P(\xi(A_n)=0) \le  \mathbb P(\xi_n(A_n)=0) \le (1+C \delta^n) \mathbb P(\xi(A_n)=0).
	$$
\end{lemma}

\begin{proof}
We begin with a standard determinantal representation for the vacuum probability $\mathbb P(\xi(A_n)=0)$ (see, e.g., \cite[Theorem 6]{adhikari2017hole}). 
For every $m \in \mathbb{N}$, the  determinantal structure of $\xi_m$  implies that 
\begin{align*}
\mathbb P(\xi_m(A_n)=0)
&=
\frac{1}{m!}
\int_{A_n^c}\dots \int_{A_n^c} \det (K_m(z_i,z_j))_{i,j=1}^m  \, dz_1\dots dz_m\\
& = 
\frac{1}{m!}\int_{A_n^c}\dots \int_{A_n^c} \det (\varphi_k(z_i) )_{i,k=1}^m \det (\overline{\varphi_k(z_i)} )_{i,k=1}^m   \, dz_1\dots dz_m\\
&= 
\frac{1}{m!}\int_{A_n^c}\dots \int_{A_n^c}    \sum_{\sigma_1, \sigma_2  \in S_m} \operatorname{sgn}(\sigma_1 \sigma_2)  
 \prod_{i=1}^m \varphi_{\sigma_1(i)}(z_i) \overline{\varphi_{\sigma_2(i)}(z_i)}  \, dz_1\dots dz_m\\
&=
\sum_{\sigma \in S_n} \operatorname{sgn}(\sigma) \prod_{i=1}^m \int_{A_n^c} \varphi_i(z) \overline{\varphi_{\sigma(i) }(z)}\, dz\\
& =\det  \left(\int_{A_n^c} \varphi_i(w) \overline{\varphi_j(w)} dw\right)_{1\le i,j\le m},
\end{align*}
where $S_n$ is the set of all permutations of $\{1,\dots,n\}$. As noted in the proof of \cite[Theorem 6]{adhikari2017hole}, this representation implies $\mathbb P(\xi_m(A_n)=0)$ is decreasing in $m$ and converges to $\mathbb P(\xi(A_n)=0)$ as $m \to \infty$, giving the first inequality of the assertion. We give the details here for completeness. 

Let \[ M_{m}(U):=\Big(\int_{U^c} \varphi_i(w) \overline{\varphi_j(w)} dw\Big)_{1\le i,j\le m},\] so that we have $\mathbb P(\xi_m(A_n)=0)=\det(M_{m}(A_n))$. 
For every $U\subset \mathbb{C}$, $M_{m}(U)$ is the integral of a positive-definite matrix function, so it too is positive-definite. 
In particular $ M_{m}(A_n)$ is positive-definite, so by Fischer's inequality \cite[Theorem 7.8.5]{horn2012matrix},
\[
\det ( M_{m}(A_n) ) \le \det ( M_{m-1,n} )  \left(  \int_{A_n^c} \varphi_m(w)   \overline{\varphi_j(w)} dw \right) \le \det ( M_{m-1,n} ),
\]
as desired.

	 For the second assertion, note that for all $m \ge n$, the Schur complement formula implies that
	\begin{align}\label{e:intstep}
		\det M_{m}(A_n)=\det (M_{n} (A_n)) \det( M_{m}(A_n)/M_{n} (A_n)),
	\end{align}
	where $M_{m}(A_n)/M_{n} (A_n)$ is the Schur complement of the block $M_{n} (A_n)$ of the matrix $M_{m}(A_n)$. 
By our assumption that $A_n \subset B_{s \sqrt n}$, we have $M_m(A_n) \succeq M_m( B_{s \sqrt n})$, since \[
M_m(A_n) - M_m( B_{s \sqrt n})  = M_m \big( A_n \cup  B_{s \sqrt n}^c  \big) \succeq  0,
\]
where $A \succeq B$ means that $A - B$ is positive-definite for $m \times m$-matrices $A,B$.
 Then, by the monotonicity of the Schur complement \cite[7.7.P41]{horn2012matrix}, 
	\begin{align*}
		\det( M_m(A_n)/M_n(A_n)) \ge  \det( M_m(B_{s\sqrt n})/M_n(B_{s\sqrt n})), \quad m \ge n.
	\end{align*}
As $B_{s \sqrt n}$ is a circular domain, we have
\[ \int_{B_{s' \sqrt n}^c} \varphi_i(w) \overline{\varphi_j(w)} \, dw =0
\]
for all $i \neq j$. Then
	$$
	\det( M_m(B_{s \sqrt n})/M_n(B_{s \sqrt n})) \ge \prod_{k=n+1}^\infty \int_{B_{s \sqrt n}^c} |\varphi_k(w)|^2\,  \d w, \quad m \ge n.
	$$
	To show that the above is lower bounded by a constant (as $n \to \infty$), we note that 
	\[
	\int_{B_{s \sqrt n}^c} |\varphi_k(w)|^2 \d w
	=  \frac{1}{ \pi (k-1)!} \int_{B_{s \sqrt n}^c} |w|^{2(k-1)} e^{-|w|^2} \, \d w 
	 =  \frac{2}{  (k-1)!}  \int_{s \sqrt{n}}^\infty r^{2k-1} e^{-r^2}\,dr
\] 
is the probability that $R_k >s \sqrt n$, where $R_k^2$ follows a $\operatorname{Gamma}(k,1)$-distribution. 
For all $\lambda>0$, let $Y_{\lambda}$ be Poisson($\lambda$)-distributed and note that $\P(R_k^2 \le  t)=\P(Y_{t}\ge  k) $. Hence,
	\begin{align*}
		\prod_{k=n+1}^\infty \int_{B_{s \sqrt n}^c} |\varphi_k(w)|^2 \d w = \prod_{k=n+1}^{\infty} \P(R_k^2 \ge s n) \ge \prod_{k=n+1}^{\infty} \P(R_k^2 \ge s k)=\prod_{k=n+1}^{\infty} (1-\P(Y_{s k} >k)).
	\end{align*}
	Next we use a well-known version of the Chernoff bound for Poisson random variables, which gives \[\mathbb P(Y_{\lambda}>k)\le k^{-k} (e\lambda)^k e^{-\lambda}\] for all $\lambda <k$. Hence,
	$$
	\P(Y_{s k} >k) \le (e^{1-s}s)^k, \quad k \in \mathbb N.
	$$
	Let $s_2:=e^{1-s }s \in (0,1)$. Since $\log (1-x) \ge 1-\frac{1}{1-x}$ for $x \in (0,1)$, we arrive at the bound
	\begin{align*}
		\prod_{k=n+1}^\infty \int_{B_{s \sqrt n}^c} |\varphi_k(w)|^2 \d w \ge \exp\left(-\sum_{k=n+1}^\infty \frac{s_2^k}{1-s_2^k}  \right) \ge \exp\left(- \frac{s_2^{n+1}}{(1-s_2^n)(1-s_2)}  \right).
	\end{align*}
	Therefore, using \eqref{e:intstep} and taking $m\rightarrow \infty$,
	$$
	\mathbb P(\xi_n(A_n)=0)=\det M_{n}(A_n) \le (1+C s_2^{n}) \mathbb P(\xi(A_n)=0),
	$$
for sufficiently large $n$ (depending only on $s_2)$. 
In the previous inequality, we used that since $(1-s_2^n) (1-s_2)$ converges to $1-s_2$, there exists $C(s)>0$ such that for sufficiently large $n$ (depending only on $s$),
\[
\exp\left(\frac{s_2^{n+1}}{(1-s_2^n)(1-s_2)}  \right)\le \exp( C s_2^{n+1} ) \le 1 + 2C s_2^{n}.
\]
This completes the proof.
\end{proof}

\begin{lemma}\label{lem:mom2}
Fix  $s<1$ and a sequence $(s_n)_{n \in \mathbb N}$ such that $s_n \to \infty$ and $s_n \le n^{2/5}$. There exists a sequence $(p_n)_{n=1}^\infty$ (depending on $s$ and  $(s_n)_{n \in \mathbb N}$) such that $p_n \to 0$ as $n \to \infty$ and 
\[
\sup_{z \in B_{s \sqrt{n}} } \P ( \xi_n ( B_{s_n} (z)) \le 2 ) \le \exp(-s_n^4(1/4+p_n))
\]
for all $n \in \mathbb{N}$.
\end{lemma}

\begin{proof}
Let $M$ be a large integer, to be determined later. Let $A_1, \dots, A_M$ be a series of annuli centered at $z$, where $A_i$ has inner radius $(i-1) s_n/M$ and outer radius $i s_n/M$. Then the annuli are disjoint, and their union is $B_{s_n}(z)$. For all $i < j  \le M$, set $D_{i,j}=  B_{s_n} (z) \setminus ( A_i \cup A_j) $. Then we have the inclusion of events
\[
\{ \xi_n ( B_{s_n} (z)) \le 2 \} \subset \bigcup_{i< j \le M}  \{ \xi_n ( D_{i,j} ) = 0\}. 
\]
By a union bound,
\[
\P ( \xi_n ( B_{s_n} (z)) \le 2) \le \sum_{i<j \le M}  \P( \xi_n ( D_{i,j} ) = 0).
\]
Let $\delta \in (0,1)$ denote the constant from  Lemma \ref{lem:gininf}.  This lemma gives
\[
\P( \xi_n ( D_{i,j} ) = 0) \le (1+ C \delta^n) \P( \xi( D_{i,j} ) = 0)
\]
for all $i  < j < m$, 
since $B_{s_{n}}(z) \subset B_{s' \sqrt{n}}(0)$ for $s ' = (s+1)/2$ and all $z \in B_{s\sqrt{n}}$, by the assumed upper bound on $s_n$. 
Therefore,
\begin{equation}\label{unionbd}
	\P ( \xi_n ( B_{s_n} (x)) \le 2) \le (1+ C \delta^n)  \sum_{i<j \le M}  \P( \xi ( D_{i,j} ) = 0).
\end{equation}

Now, since infinite Ginibre process is stationary, we may suppose that each $D_{i,j}$ is centered at the origin and use the representation  stated in Lemma~\ref{l:gincounting}. Let $(X_k)_{k \in \mathbb{N} }$ be an infinite collection of independent random variables such that $X^2_k\sim \operatorname{Gamma}(k,1)$. By Lemma~\ref{l:gincounting},
\[
\P( \xi ( D_{i,j} ) = 0)  = \P(  \#\{ k \in \mathbb N : X_{k} \in |D_{i,j} | \} =0 ) =  \P (  X_{k} \notin |D_{i,j}| \, \forall k ) =\prod_{k=1}^\infty \P( X_{k} \notin |D_{i,j}|) ,
\]
where $|D_{i,j}|$ is the projection of $D_{i,j}$ to the $r$-axis in the $(r,\theta)$ plane. Inserting this into \eqref{unionbd}, we obtain 
\begin{align}
\begin{split}
\label{e:annulibound}
	\P ( \xi_n ( B_{s_n} (x)) \le 2)
	&\le (1+C\delta^n) \sum_{i < j \le M} 
	\prod_{k=1}^\infty \P( X_{k} \notin |D_{i,j}|)\\
	&\le 
	(1+C\delta^n) \sum_{i < j \le M} 
	\prod_{k=1}^\infty
	\left( \P(X_{k} \in |A_i|) +  \P(X_{k} \in |A_j|) +\P(X^2_{k} > s_n^2) \right).
\end{split}
\end{align}
As observed in \eqref{e:gammadensitybound}, the density of $X^2_{k}$ is uniformly bounded by 1 for all $k \in \mathbb N$. Hence, for all $i$ such that $1 \le i \le M$, 
\begin{align*}
	\P(X_{k}  \in |A_i|)&=\P\Big( \big((i-1)s_n/M\big)^2\le  X^2_{k} \le (is_n/M)^2\Big) \le 
	(2i-1) (s_n/M)^2\le 2s_n^2/M.
\end{align*}
The moment generating function of $X^2_{k}$ is well known, and we have for any $t < 1$ that 
\[
\P( X^2_{k} \ge s_n^2) \le e^{ - t s_n^2} \E [ e^{ t X^2_{k}} ] \le e^{ - t s_n^2}  ( 1 -  t)^{-k}.
\]
Optimizing this bound by taking $t = 1 -  k/s_n^2$ gives 
\[\P(X^2_{k} > s_n^2)\le \exp\big(-s_n^2 +k - k\log ( k/s_n^2)\big).\] 
Hence, with $M:=e^{s_n^3}$ we obtain from \eqref{e:annulibound} that 
\[
\P ( \xi_n ( B_{s_n} (x)) \le 2) \le 
(1+C \delta^n) e^{2 s^3_n} \prod_{k=1}^{s_n^2 } \left ( 4 s_n e^{ - s_n^3}  +  \exp\left(-s_n^2+k-  k\log ( k /s_n^2)\right) \right).
\]
For sufficiently large $n$, depending on the choice of  $(s_n)_{n \in \mathbb N}$, we have  uniformly for all $k \le s_n^2$ that
\[
4 s_n e^{ - s_n^3}  +  \exp\left(-s_n^2+k-2k\log (k /s_n^2)\right) \le  \exp\left(- s_n^2 +k - k\log (k /s_n^2) +s_n^{-5/2}\right).
\]
 Then, bounding $(1 + C \delta^n) \le e^{s_n^3}$, we have 
\begin{align*}
\P ( \xi_n ( B_{s_n} (x)) \le 2) &\le
	e^{3 s^3_n} \prod_{k=1}^{s_n^2} 
	\exp\left(-s_n^2 + k - k\log (k /s_n^2)+s_n^{-5/2}\right) \\
	& \le \exp\left(- s^4_n ( 1/4 + p_n ) \right).
\end{align*}
In the last inequality, we used the elementary identity
\[
\sum_{k=1}^{s_n^2}   k \big(1 +   \log (s^2_n)  \big)  
 =
\big(1 +   \log (s^2_n)  \big)    \frac{ s_n^2  (s_n^2 + 1 ) }{2}, 
\]
and 
\[
\sum_{k=1}^{s_n^2}   k \log k   = \frac{ s_n^4}{2} \log(s_n^2)  - \frac{ s_n^4}{4} + \frac{s_n^4 }{2} \log (s_n^2) + O\big ( \log (s_n^2)\big),
\]
by the Euler--Maclaurin formula. 
This concludes the proof.
\end{proof}

Now we can deduce the asymptotic scaling of the sequence $(r_n)_{n \in \mathbb N}$. We recall that $r_n$ depends on a parameter $\kappa>0$ that we suppress in the notation.

\begin{lemma} \label{lem:int}
For every Borel set $B$ with $\sup_{z \in B} |z|<1$, and all $\kappa>1$, we have
	\begin{align*}
		\sup_{z \in \sqrt n B}\Big|\frac{r_n(z)^4}{4 \log n} -1\Big|\to 0\quad \text{as }n \to \infty.
	\end{align*}
\end{lemma}

\begin{proof}
From the domination property 
\eqref{Palmdom},
 Lemma \ref{lem:gininf}, and stationarity of the infinite Ginibre process $\xi$, we have that for all $z \in \sqrt n B$,
\begin{align}\label{e:combine}
	\P(\xi_n^{z!}(B_{r_n(z)}(z))=0) \ge 	\P(\xi_n(B_{r_n(z)}(z))=0) \ge 	\P(\xi(B_{r_n(z)}(0))=0).
\end{align}
Moreover, by 
	Lemma~\ref{l:requality}, there exists $C(B,\kappa) > 1$ such that for $n \ge C$, 
\begin{align}\label{e:rdefrecall}
	n\P(\xi_n^{z!}(B_{r_n}(z))=0) \rho_n(z) =\kappa ,\quad z \in B_{\sqrt n}(0),
\end{align}
	where $\rho_n(z)=K_n(z,z)$ is the intensity function of $\xi_n$. Since $  \rho_n( z_n) \uparrow \frac 1 \pi$ uniformly as $n \to \infty$ for all sequences $(z_n)_{n=1}^\infty$ such that $z_n \in \sqrt{n}B$, we have for $n$ sufficiently large (depending only on the choice of $B$) that 
	$$
	\frac 1{2\pi  } \le \rho_n(z) \le \frac 1{\pi },\quad z \in \sqrt{n} B.
	$$
	In combination with \eqref{e:combine} and \eqref{e:rdefrecall}, this gives
\begin{align}
\frac{\pi c }{n}
\le
\inf_{z \in \sqrt n B} \P(\xi(B_{ r_n(z)}(0))=0) 
 \qquad 
\sup_{z \in \sqrt n B} \P(\xi(B_{ r_n(z)}(0))=0) \le \frac{2\pi c}{n}.\label{eq:nginbou}
\end{align}
Moreover, from \cite[Proposition 7.2.1]{BKPV09} we have
\begin{align}
\lim_{r\rightarrow \infty} \frac{1}{r^4}  \log\P(\xi (B_{r}(0))=0) = - \frac{1}{4}. \label{eq:ginas}
\end{align}
Define the sequence 
$(v_n)_{n \in \mathbb N}$
by  $v_n = \inf_{z \in \sqrt n B} r_n(z)$.
Using
\[\P(\xi(B_{v_n}(0))=0) =
\sup_{z \in \sqrt n B} \P(\xi(B_{ r_n(z)}(0))=0), \]
together with \eqref{eq:nginbou} and \eqref{eq:ginas}, we find that $(v_n)_{n \in \mathbb N}$ is unbounded.
Then  \eqref{eq:nginbou} and \eqref{eq:ginas} yield
\begin{align}
		\liminf_{n \to \infty} \inf_{z\in \sqrt nB}\frac{r_n(z)^4}{4 \log n}\ge -	\liminf_{n \to \infty} \frac{\inf_{z\in \sqrt nB} r_n(z)^4}{4 \log  \P(\xi(B_{v_n}(0))=0) + 4 \log 2\pi c }  =  1.\label{e:suplower}
\end{align}
	
To derive the complementary bound, we first estimate the vacuum probability under the reduced Palm measure $\xi_n^{z!}$. 
Using the coupling of $\xi_n$ and $\xi_n^{z!}$ introduced in Section~\ref{sec:coup}, we have
\begin{equation} \label{e:couplingbound2}
	\mathbb P(\xi_n^{z!}(B_{r_n}(z))=0) \le \mathbb P(\xi_n(B_{r_n}(z)) \le 1).
\end{equation}
Define the sequence 
$(\tilde v_n)_{n \in \mathbb N}$
by  $\tilde v_n = \sup_{z \in \sqrt n B} r_n(z)$. By Lemma~\ref{lem:mom2}, 
	$$
	\inf_{z \in \sqrt n  B} \P(\xi_n(B_{r_n}(z))\le 1) \le \exp\Big(-\tilde v_n^4  (1/4+p_n)
	\Big),
	$$
	where $p_n \to 0$ as $n \to \infty$. 
In combination with \eqref{e:combine},  \eqref{eq:nginbou}, and \eqref{e:couplingbound2}, this gives 
\begin{equation}\label{e:supupper}
\limsup_{n \to \infty} \sup_{z \in \sqrt n B} \frac{r_n(z)^4}{4 \log n}\le 1.
\end{equation}
Combining \eqref{e:suplower} and \eqref{e:supupper} completes the proof.
\end{proof}

\section{Proofs of Main Results}\label{sec:pr1}

In the next proof, we will need a slight modification of \cite[Theorem 3.1]{BSY21}, which provides an estimate on the KR distance between two point processes that is based on Stein's method. 
Since we are unable to verify that $\E[(\Xi_n[\xi_n]\Delta \Xi_n[\xi_n^{z!, \varnothing}])(B)]$ is measurable as a function of $z$ for an arbitrary $B \subset B_1(0)$, we cannot apply \cite[Theorem 3.1]{BSY21} directly to obtain the desired result.  Instead, we adapt the arguments given there to our setting.

\begin{lemma}\label{l:august}
Fix $\kappa >1$, and let  $\zeta_\kappa$ be a Poisson process with intensity measure $L_n$.  Fix $t \in (0,1)$. There exists a constant $C(t, \kappa)>1 $ such that for all $B \subset B_t(0)$ and $n \ge C$, we have 
	\begin{align}
	\begin{split}\label{eq:KRbou}
		&{d_{\mathrm{KR}}}(\Xi_n \cap B,\zeta_\kappa \cap B) \\ & \le  2\Big\{ \int_{B} \E\big[\Xi_n(T_{n,z}) \big] L_n(\d z)+ \int_{B}\E\big[\Xi_n^{z!}(T_{n,z}) \big] \, L_n(\d z) + L_n(B) \sup_{z \in B}\E \big[(\tilde\Xi_n \Delta \tilde\Xi_n^{z!}) (S_{n,z}) \big] \Big\},
		\end{split}
	\end{align}
where $T_{n,z}=B_{\log n/\sqrt n}(z)$, $S_{n,z} = B \setminus T_{n,z}$, and for any $A\subset\mathbb{C}$,
\[ \tilde \Xi_n (A) := \Xi_n\left[\xi_n|_{B_{r_{n}}(\sqrt nz)^c} \right] (A \cap {S_{n,z}} ), 
\qquad \tilde \Xi_n^{z!}(A):= \Xi_n\left[
\xi_n^{\sqrt n z, \varnothing}|_{B_{r_{n}}(\sqrt nz)^c}
 \right] ( A \cap {S_{n,z}})
 \]
 with $\xi_n|_{B_{r_{n}}(\sqrt nz)^c}$ and $\xi_n^{\sqrt n z, \varnothing}|_{B_{r_{n}}(\sqrt nz)^c}$ coupled as in  Section~\ref{sec:coup}.\footnote{We recall that the construction of this coupling used Lemma~\ref{l:stocdom}, which requires that $n$ be sufficiently large in a way that depends only on $\kappa$ and $g$; this requirement is absorbed into the constant $C(t,\kappa)$ in the statement of this lemma.}
\end{lemma}

\begin{proof}
Let $\mathcal L$ be the generator of the Glauber dynamics for the Poisson process $\zeta_\kappa$ (see \cite[(2.7)]{BSY21}), which is given by
\[
\mathcal L h(\omega) = \int_{\mathbb C} D_z h(\omega) L_n( \d z) - \int_{\mathbb C} D_z h(\omega - \delta_z ) \omega( \d z),
\]
where $D_z h(\omega) = h(\omega + \delta_z) - h(\omega)$ and $\mathcal L$ acts on functions  $h\colon \widehat{\mathbf{N}} \rightarrow \R$. Let $P_s$ be the Markov semigroup associated with $\mathcal L$.
The proof of \cite[Theorem 3.1]{BSY21}, specialized to our setting, yields
\begin{align}
\begin{split}
{d_{\mathrm{KR}}}(\Xi_n\cap B, \zeta_\kappa \cap B) &\le  \sup_{h \in \text{Lip}}  \int_0^\infty \big |\E[ \mathcal L P_s h(\Xi_n \cap B)] \big| \d s\\ 
	&= \sup_{h \in \text{Lip}}   \int_0^\infty \left| \int_B \left( \E \big[ D_z P_s h(\Xi_n \cap B )\big] - \E \big[D_z P_s h(\Xi_n^{z!} \cap B)\big] \right)  L_n(\d z)   \right| \d s.
\label{e:poissoncoupling}
\end{split}
\end{align}
In the following we will discuss how to bound the right-hand side of \eqref{e:poissoncoupling}. This will involve a coupling outside of a small ball around an arbitrary $z\in B$ and a rough union bound within this ball. 
By the definition of $D_z$, we have 
\begin{align}
\begin{split}
& \Big| \E \big[ D_z P_s h(\Xi_n \cap B )\big] - \E \big[D_z P_s h(\Xi_n^{x!} \cap B)\big] \Big| \\ 
 &= 
\left| \E\left[  P_s h( \Xi_n \cap B  + \delta_z )  -  P_s h( \Xi_n^{z!}  \cap B  + \delta_z ) 
 - 
  P_s h( \Xi_n \cap B  )  +  P_s h( \Xi_n^{z!}  \cap B  )  \right]\right|.
\label{e:split1}
\end{split}
\end{align}
By \cite[(2.9)]{BSY21},
\begin{align*}
\E\left[  \big| P_s h( \Xi_n \cap B  )  -  P_s h( \Xi_n \cap  S_{n,z}  )  \big|  \right] &\le e^{-s} \E \big[\Xi_n ( T_{n,z})\big],
\\
\E\left[  \big| P_s h( \Xi^{z!}_n \cap B  )  -  P_s h( \Xi^{z!}_n \cap S_{n,z}  )  \big|  \right] &\le e^{-s} \E \big[\Xi^{z!}_n ( T_{n,z} )\big],  \\
\E\left[  \big| P_s h( \Xi_n \cap B  + \delta_z )  -  P_s h( \Xi_n \cap S_{n,z}  + \delta_z)  \big|  \right] &\le e^{-s} \E \big[\Xi_n ( T_{n,z})\big],
\\
\E\left[  \big| P_s h( \Xi^{z!}_n \cap B  +\delta_z )  -  P_s h( \Xi^{z!}_n  \cap S_{n,z} +\delta_z )  \big|  \right] &\le e^{-s} \E \big[\Xi^{z!}_n ( T_{n,z} )\big].
\end{align*}
Inserting these bounds in \eqref{e:split1} gives 
\begin{align}
\begin{split}
& \Big| \E \big[ D_z P_s h(\Xi_n \cap B )\big] - \E \big[D_z P_s h(\Xi_n^{z!} \cap B)\big] \Big| \\  
 &\le \left| \E \left[ P_s h( \Xi_n \cap S_{n,z} )  -  P_s h( \Xi_n^{z!}  \cap S_{n,z} )  \right]\right| \\
  &  \quad + 
\left| \E\left[  P_s h( \Xi_n \cap S_{n,z}  + \delta_z )  -  P_s h( \Xi_n^{z!} \cap S_{n,z} + \delta_z ) \right] \right|  +2  e^{-s} \left(\E \big[\Xi_n ( T_{n,z} )\big] + \E \big[\Xi^{z!}_n ( T_{n,z} )\big]\right). 
\label{e:split2}
\end{split}
\end{align}

Note that $\Xi_n[\omega] \cap S_{n,z} $ depends on $\omega$ only through
\[ \omega \cap \big(\sqrt nS_{n,z} \oplus B_{r_0}(0) \big), \qquad r_0:=\sup_{w \in \sqrt n S_{n,z}} r_n(w),\]
where $\oplus$ denotes the Minkowski sum. 
 By Lemma \ref{lem:int}, there exists $C(t,\kappa)> 1 $ such  that \[ \sqrt nS_{n,z} \oplus B_{r_0} \subset B_{r_n(z)}(\sqrt n z)^c\]  for $n \ge C$. Then by \cite[(2.9)]{BSY21} and the definitions of $\tilde \Xi_n$ and $\tilde \Xi_n^{z!}$ in the statement of the lemma, 
\begin{align}\label{e:measureremark}
\begin{split}
 \left| \E \left[ P_s h( \Xi_n  \cap S_{n,z} )  -  P_s h( \Xi_n^{z!}   \cap S_{n,z} )  \right]\right| 
& \le e^{-s} \E[(\tilde\Xi_n \Delta \tilde\Xi_n^{z !}) (S_{n,z})]  \le \sup_{z \in  B}  e^{-s} \E[(\tilde\Xi_n \Delta \tilde\Xi_n^{z !}) (S_{n,z})], \\
 \left| \E\left[  P_s h( \Xi_n  \cap S_{n,z}  + \delta_z )  -  P_s h( \Xi_n^{z!}    \cap S_{n,z} + \delta_z ) \right] \right|
& \le e^{-s} \E[\tilde(\Xi_n \Delta \tilde\Xi_n^{z !}) (S_{n,z})] \le  \sup_{z \in  B}  e^{-s} \E[(\tilde\Xi_n \Delta \tilde\Xi_n^{z !}) (S_{n,z})] .
\end{split}
\end{align}
Inserting this into \eqref{e:split2}, and integrating over $s$ in \eqref{e:poissoncoupling}, we obtain the conclusion. 
\end{proof}

\begin{remark}
In \eqref{e:measureremark}, we took suprema over $ B$ to ensure that our upper bounds are measurable as functions of the $z\in B$ appearing on the differences on the left-hand sides of the inequalities (and hence can be integrated with respect to $L_n(\d z)$). These suprema are constants in $z$, and hence trivially measurable. In contrast, it does not seem apparent from the construction of our couplings that $z \mapsto \E[(\tilde\Xi_n \Delta \tilde\Xi_n^{z !}) (S_{n,z})]$ is measurable. 
\end{remark}
We are now ready for the proofs of our main results.
\begin{proof}[Proof of Theorem \ref{th}]
		Fix a Borel set $B \subset B_1(0)$ with $\sup_{z \in B} |z|<s$. 
We have already proved \eqref{e:rcnlimit} in Lemma~\ref{lem:int}, so we turn to \eqref{e:KRthm}.

We use \eqref{eq:KRbou} and bound the three terms in the curly brackets separately.
Using the definitions of $L_n$ and $\Xi_n$, the first term is 
\begin{align*}
\int_{B}  \E[\Xi_n(T_{n,z})] L_n(\d z) &=\int_{\sqrt nB} \int_{\sqrt nT_{n,z}} \P(\xi_n^{z!}(B_{r_n}(z))=0) \P(\xi_n^{w!}(B_{r_n}(w))=0) \rho_n(z) \rho_n(w) \, \d w \, \d z.
\end{align*}
Since $n\P(\xi_n^{z!}(B_{r_n}(z))=0) \rho_n(z)=\kappa$ for all $z \in B$ by the definition of $r_n$, this integral equals
\[
\int_{B}  \E[\Xi_n(T_{n,z})] L_n(\d z)  = 
\kappa^2 |B| |T_{n,z}|=\frac{\pi \kappa^2 |B| (\log n)^2}{n}.
\]

Next, we consider the second term in the curly brackets in \eqref{eq:KRbou}. By Remark~\ref{r:XiPalm} and \eqref{defPalm0},
\begin{align}
	&\int_{B}  \E[\Xi_n^{z!}(T_{n,z})] L_n(\d z)\notag \\
	&=\int_{B}  \E[\Xi_n[\xi_n^{z,\varnothing}](T_{n,z}\setminus \{z\})] L_n(\d z)\notag \\
	&=\E\Big[\int_{\sqrt n B} \int_{\sqrt n T_{n,z}} \one\{(\xi_n-\delta_w)(B_{r_n}(w)) =0\} \one\{(\xi_n-\delta_z)(B_{r_n}(z)) =0\} (\xi_n-\delta_z) (\d w)\,  \xi_n(\d z)\Big].\label{e:prev1}
\end{align}
Now we use that $(\xi_n-\delta_w)(B_{r_n}(w)) =0$ and $(\xi_n-\delta_z)(B_{r_n}(z)) =0$ if and only if both 
\[ (\xi_n-\delta_w-\delta_z)(B_{r_n}(w) \cup B_{r_n}(z)) =0, \qquad |z-w|\ge \max\{r_n(w),r_n(z)\}\]
hold.  Therefore, by using \eqref{defPalm} twice in succession, the expectation in \eqref{e:prev1} equals 
\begin{align}\label{e:2ndterm}
\int_{\sqrt n B} \int_{\sqrt n T_{n,z}} \P(\xi_n^{z!, w!}(B_{r_n}(z) \cup B_{r_n}(w))=0)\one\{|z-w|\ge \max\{r_n(w),r_n(z)\}\}  \rho_n^{(2)}(z,w) \, \d w \, \d z,
\end{align}
where we recall from Section~\ref{sec:coup} that $\xi_n^{z!, w!}:=(\xi_n^{z!})^{w!}$. Since the reduced Palm process of a determinantal process is again determinantal (see \eqref{KPalm}), $\xi_n^{z!, w!}$ is a determinantal process and therefore has negative associations. Assume without loss of generality that $r_n(w) \ge r_n(z)$. Then by \eqref{defNA} (applied to indicator functions for the disjoint sets $B_{r_n}(z)$ and $B_{r_n}(w)\setminus B_{r_n}(z)$),
\begin{equation}\label{e:BTintegral}
  \P(\xi_n^{z!, w!}(B_{r_n}(z) \cup B_{r_n}(w))=0) \le \P(\xi_n^{z!, w!}(B_{r_n}(z))=0) \P(\xi_n^{z!, w!}(B_{r_n}(w) \setminus B_{r_n}(z))=0).
\end{equation}
Using the coupling $(\xi_n^{z!},\xi_n^{z!,w!})$ given in Section~\ref{sec:coup}, we bound the right-hand side of \eqref{e:BTintegral} by 
\[
 \P(\xi_n^{z!}(B_{r_n}(z))\le 1) \P(\xi_n^{z!}(B_{r_n}(w) \setminus B_{r_n}(z))\le 1).
\]
With the same idea applied to $(\xi_n, \xi_n^{z!})$ (again using the coupling specified in Section~\ref{sec:coup}), we have
\begin{equation}\label{e:exclambounds}
 \P(\xi_n^{z!}(B_{r_n}(z))\le 1) \P(\xi_n^{z!}(B_{r_n}(w) \setminus B_{r_n}(z))\le 1) \le \P(\xi_n(B_{r_n}(z))\le 2) \P(\xi_n(B_{r_n}(w) \setminus B_{r_n}(z))\le 2).
\end{equation}
Fix $\epsilon>0$. We obtain from Lemma \ref{lem:mom2} and Lemma \ref{lem:int} that there exists $C(s, \epsilon, \kappa) > 1$ such that for all $n \ge C$, 
\[\P(\xi_n(B_{r_n}(z))\le 2) \le  n^{\varepsilon/2-1}, \qquad \P(\xi_n(B_{r_n}(w) \setminus B_{r_n}(z))\le 2)\le n^{\varepsilon/2-1/16},\] where the second inequality follows from the inclusion \[B_{r_n(w)/2}\left( w+\frac {w-z}{2}\right) \subset B_{r_n}(w) \setminus B_{r_n}(z),\] since $r_n(w) \ge r_n(z)$ by assumption. Since $\rho_n^{(2)} \le 1/\pi^2$ (by a direct calculation using the explicit form of the kernel),  
we conclude from \eqref{e:prev1}, \eqref{e:2ndterm},  \eqref{e:BTintegral}, and \eqref{e:exclambounds} that
$$
	\int_{B}  \E[\Xi_n^{z!}(T_{n,z})] L_n(\d z) \le \frac{n^{\epsilon-1/16}|\sqrt{n} B||\sqrt nT_{n,0}|}{\pi^2}=\frac{n^{\epsilon-1/16}(\log n)^2|B|}{\pi}.
$$

Finally, we consider the integrated expected symmetric difference term on the right-hand side of \eqref{eq:KRbou}.  
Recall that $S_{n,z}:=B \setminus T_{n,z}$ for $z \in B_1(0)$. From the definitions of $\tilde\Xi_n$ and $\tilde\Xi_n^z$ in the statement of Lemma~\ref{l:august} (and the facts noted following \eqref{e:split2}), we obtain that almost surely,
\begin{align}
(\tilde\Xi_n \Delta \tilde\Xi_n^z) (S_{n,z})&\le (\xi_n \Delta  \xi_n^{\sqrt nz,\varnothing}) (\sqrt n S_{n,z})\label{eq:symm1}\\ &\quad+  \sum_{w \in \xi_n \cap  \xi_n^{\sqrt n z,\varnothing}\cap (\sqrt nS_{n,z})} \big|\one \{ \xi_n^{\sqrt nz,\varnothing}(B_{r_n}(w))=0\}-\one \{\xi_n(B_{r_n}(w))=0\} \big|.\label{eq:symm2}
\end{align}

We begin by bounding the expectation of the right-hand side of \eqref{eq:symm1}. From the coupling indicated in Section~\ref{sec:coup}, we have 
\[ \xi_n^{\sqrt n z!}  |_{B_{r_n}(\sqrt n z)^c}  \subset \xi_n |_{B_{r_n}(\sqrt n z)^c}, \qquad \xi_n^{\sqrt n z!} |_{B_{r_n}(\sqrt n z)^c}\subset  \xi_n^{\sqrt{n} z!,\varnothing} |_{B_{r_n}(\sqrt n z)^c}. \]  Using these inclusions in combination with  $B_{r_n}(\sqrt n z) \subset \sqrt nT_{n,z}$, we have 
\begin{align}
	\E [(\xi_n \Delta  \xi_n^{\sqrt n z,\varnothing}) (\sqrt nS_{n,z})]&\le \E[(\xi_n \setminus \xi_n^{\sqrt n z!})(\sqrt nS_{n,z})] +\E[( \xi_n^{\sqrt{n} z!,\varnothing} \setminus \xi_n^{\sqrt n z!})(\sqrt nS_{n,z})] \nonumber\\
	&=\E [\xi_n (\sqrt nS_{n,z})]- \E[ \xi_n^{\sqrt n z!}(\sqrt n S_{n,z})] \label{etabou2a}\\
	&\quad +\E[ \xi_n^{\sqrt n z!,\varnothing }(\sqrt nS_{n,z})]-\E[\xi_n^{\sqrt n z!}(\sqrt nS_{n,z})].\label{etabou2b}
\end{align}
Using Lemma~\ref{l:Kbound}, \eqref{KPalm}, and  $K_n(z,z)=\rho_n(z)\ge \frac 1{2\pi}$ uniformly for all $|z| < s$  for $n \ge C(s)$, we can bound the difference in \eqref{etabou2a} for every $z \in B$ by
\begin{align}
\E [\xi_n (\sqrt nS_{n,z})]- \E[ \xi_n^{\sqrt n z!}(\sqrt n S_{n,z})] & \le 2 \pi \int_{\sqrt nS_{n,z}} |K_n(\sqrt nz,w)|^2\,  \d w \notag  \\ &\le 2 \pi |B| n \left(  \sup_{w \in \sqrt nS_{n,z}} e^{-|w-\sqrt n z|^2/2}+O(e^{-cn}) \right)\nonumber\\
&\le \frac{2 \pi |B|}{{n^2}}+O(e^{-cn}).
\label{MRdens}
\end{align}

 To bound \eqref{etabou2b},  we use that $L_n(\d z)= \kappa\, \d z$ for all $z \in B_0(s)$ and $n \ge C(s,\kappa)$ by the definition of $L_n$ and Lemma~\ref{l:requality}. Together with \eqref{defPalm} and \eqref{defPalm0}, this gives for all Borel sets $A \subset B_1(0)$ that 
\begin{align*}
	&n\int_A \E[\xi_n^{\sqrt n z!,\varnothing}(\sqrt{n}S_{n,z})] \P(\xi_n^{\sqrt{n}z!}(B_{r_n}(\sqrt{n}z))=0)  \rho_n(\sqrt{n}z) \d z\\
	&\quad=\E\Big[\int_{\sqrt{n} A}(\xi_n-\delta_z)(\sqrt{n}S_{n,z/\sqrt{n}}) \one\{(\xi_n-\delta_z)(B_{r_n}(z))=0\} \xi_n(\d z) \Big]\\
	&\quad=\int_{\sqrt n A} \E\Big[\xi_n^{z!} (\sqrt n S_{n,z/\sqrt n}) \one\{\xi_n^{z!}(B_{r_n}(z))=0\}\Big] \rho_n(z) \d z\\
	&\quad=n\int_{A} \E\Big[\xi_n^{\sqrt n z!} (\sqrt n S_{n,z}) \one\{\xi_n^{\sqrt n z!}(B_{r_n}(\sqrt n z))=0\}\Big] \rho_n(\sqrt n z) \d z,
\end{align*}
where the final equality follows from a change of variables. 
Since $A$ is arbitrary, the integrands must agree for $L_n$-almost all $z \in B$.
Hence, we find that \eqref{etabou2b} becomes 
\begin{align}
	&\E[ \xi_n^{\sqrt n z!,\varnothing}(\sqrt nS_{n,z})]-\E[\xi_n^{\sqrt nz!}(\sqrt nS_{n,z})]\nonumber\\
	&\quad={\kappa^{-1}n \rho_n(\sqrt n z)\P(\xi_n^{\sqrt n z!}(B_{r_n}(\sqrt nz))=0)}\Big(\E[ \xi_n^{\sqrt n z!,\varnothing}(\sqrt nS_{n,z})]-\E[\xi_n^{\sqrt nz!}(\sqrt nS_{n,z})]\Big)\nonumber\\
	&\quad={\kappa^{-1}n \rho_n(\sqrt n z)}\Big(\E \big[\xi_n^{\sqrt nz!}(\sqrt nS_{n,z}) \one\{\xi_n^{\sqrt nz!}(B_{r_n}(\sqrt nz))=0\} \big]\notag \\ & \qquad \qquad \qquad \qquad \qquad  
	-\E \big [\xi_n^{\sqrt nz!} (\sqrt nS_{n,z}) \big]\P(\xi_n^{\sqrt nz!}(B_{r_n}(z))=0)\Big)\nonumber\\
	&\quad={\kappa^{-1}n \rho_n(\sqrt n z)}\Cov\Big(\xi_n^{\sqrt nz!} (\sqrt nS_{n,z}), \one\big\{\xi_n^{\sqrt nz!}\big(B_{r_n}(\sqrt nz)\big)=0\big\} \Big). \label{covbou2}
\end{align}
For $k \in \mathbb{N}$ we consider the auxiliary functions
\begin{align*}
	f^{(k)}(\omega)&:=\min\{k,\omega(B_{r_n}(\sqrt nz))+ \one \{\omega(B_{r_n}(\sqrt nz))=0\}\},\\
	f(\omega)&:=\omega(B_{r_n}(\sqrt nz))+\one\{\omega(B_{r_n}(\sqrt nz))=0\},\quad \omega\in \mathbf{N}.
\end{align*}
It is easy to see that $f^{(k)},\, k \in \mathbb{ N},$ and $f$ are bounded and increasing. Now we use that the reduced Palm process $\xi_n^{\sqrt nz!}$ is itself  a determinantal process and therefore has negative associations (see \eqref{defNA}). Since $r_n\le \log n$ for $n \ge C(s, \kappa)$, by Lemma \ref{lem:int}, we obtain for such $n$ that 
\begin{align*}
	\Cov(\min\{k,\xi_n^{\sqrt nz!}(\sqrt nS_{n,z})\}, f^{(k)}(\xi_n^{\sqrt nz!}))\le 0, \quad k \in \mathbb N. 
\end{align*}
Hence, by monotone convergence,
\begin{align*}
	&\Cov(\xi_n^{\sqrt nz!}(\sqrt nS_{n,z}),\xi_n^{\sqrt nz!}(B_{r_n}(\sqrt nz)))+\Cov(\xi_n^{\sqrt nz!}(\sqrt nS_{n,z}), \one \{\xi_n^{\sqrt nz!}(B_{r_n}(\sqrt nz))=0\})\nonumber\\
	&\quad=\Cov(\xi_n^{\sqrt nz!}(\sqrt nS_{n,z}), f(\xi_n^{\sqrt nz!}))\\
	&\quad =\lim_{k \to \infty} \Cov(\min\{k,\xi_n^{\sqrt nz!}(\sqrt nS_{n,z})\}, f^{(k)}(\xi_n^{\sqrt nz!})) \le 0.
\end{align*}
This shows that \eqref{covbou2} is bounded by
	\begin{align}
		&-\kappa^{-1}n \rho_n(\sqrt n z)\Cov(\xi_n^{\sqrt nz!}(\sqrt nS_{n,z}),\xi^{\sqrt nz!}(B_{r_n}(\sqrt nz)))\nonumber\\
		&\quad \le \kappa^{-1}n \rho_n(\sqrt n z) \big|\E [\xi_n^{\sqrt nz!} (\sqrt nS_{n,z})] - \E[\xi_n(\sqrt nS_{n,z})]\big| \E[\xi_n^{\sqrt nz!}(B_{r_n}(\sqrt nz))]	\label{covxixi1a}\\
		&\qquad+ \kappa^{-1}n \rho_n(\sqrt n z) \big|\E[\xi_n^{\sqrt nz!}(\sqrt nS_{n,z})\xi_n^{\sqrt nz!}(B_{r_n}(\sqrt{n} z))] - \E[\xi_n(\sqrt nS_{n,z})]\E[\xi_n^{\sqrt nz!}(B_{r_n}(\sqrt nz))]\big|.
		\label{covxixi1b}
	\end{align}
	To bound \eqref{covxixi1a}, we combine the estimate  \eqref{MRdens} for the difference in absolute value signs with the bound \[\E[\xi_n^{\sqrt nz!}(B_{r_n}(\sqrt nz))] \le \E[\xi_n(B_{r_n}(\sqrt nz))] \le r_n(z)^2,\] 
where the first inequality follows from \eqref{Palmdom} and the second follows from the elementary bound $\rho_n \le \pi^{-1}$. 
Hence, by Lemma \ref{lem:int} and the bound $\rho_n(\sqrt n z) \le \frac 1 \pi$,  we have using \eqref{MRdens} that for all $z \in B$, 
\begin{align}
&\kappa^{-1}n \rho_n(\sqrt n z) 	\big|\E [\xi_n^{\sqrt nz!} (\sqrt nS_{n,z})] - \E[\xi_n(\sqrt n S_{n,z})]\big| \E[\xi_n^{\sqrt nz!}(B_{r_n}(\sqrt nz))]\nonumber\\
&\quad  \le{ \frac{4  |B|\sqrt{\log n}}{cn}}+O(e^{-cn}).\label{eq:cont2}
\end{align}

	Next we consider \eqref{covxixi1b}. We write $\rho_{n,z}^{(m)}$ for the $m$-th correlation function of $\xi_n^{z!}$. 
	By \cite[Lemma 6.4]{ST03},
	\begin{equation}\label{e:st03}
	\rho_{n,z}^{(2)}(w,v)\rho_n(z)  = \rho_n^{(3)}(z,w,v) , \qquad \rho_{n,z}(w)\rho_n(z) = \rho_n^{(2)}(z,w).
	\end{equation}
	 Using the definitions of $\xi^{z!}$ and $L_n$ together with \eqref{e:st03}, we find for $L_n$-almost all $z \in B$, 
	 	\begin{align*}
	 	& \kappa^{-1}n \rho_n(\sqrt n z)\E[\xi_n^{\sqrt nz!}(\sqrt nS_{n,z})\xi_n^{\sqrt nz!}(B_{r_n}(\sqrt n z))] - \E[\xi_n(\sqrt nB\setminus \sqrt nT_{n,z})]\E[\xi_n^{\sqrt nz!}(B_{r_n}(\sqrt nz))]\\
	 	&\quad = \kappa^{-1}n \rho_n(\sqrt n z) \int_{B_{r_n}(\sqrt n z)}\int_{\sqrt n S_{n,z}}  (\rho_{n,\sqrt n z}^{(2)}(w,v) -  \rho_{n,\sqrt n z}(w) \rho_n(v)) \,\d v\,  \d w \\ 
	 	&\quad=\kappa^{-1}n\int_{B_{r_n}(\sqrt n z)}\int_{\sqrt n S_{n,z}} \rho_n^{(3)}(\sqrt n z,w,v) -  \rho_n^{(2)}(\sqrt n z,w) \rho_n(v) \,\d v\,  \d w .
	 \end{align*}
	 	Here we apply \eqref{eq:rhodec} to the difference in the integral. Uniformly for $z\in B$, this gives by Lemma \ref{lem:int},
	 \begin{align}
	 	&\kappa^{-1}n \left| \int_{B_{r_n}(\sqrt n z)}\int_{\sqrt n S_{n,z}} \rho_n^{(3)}(\sqrt n z,w,v) -  \rho_n^{(2)}(\sqrt n z,w) \rho_n(v) \,\d v\,  \d w \right| \nonumber\\ 
	    &\quad \le \kappa^{-1} 3^{5/2}\pi |B| n^2 \sup_{z \in \sqrt n B} r_n(z)^2 \sup_{\substack{v,w \in \sqrt nB\\ |v-w| \ge \log n-r_n(\sqrt n z)}} |K_n(w,v)|\nonumber\\
	 	&\quad  \le  \kappa^{-1} 3^{5/2} \pi |B| n^2 2 \sqrt{\log n} \sup_{\substack{v,w \in \sqrt nB\\ |v-w| \ge \log n-2 \sqrt {\log n}}} e^{-|w-v|^2/2}+ O(e^{-cn})\le \frac{2 \cdot 3^{5/2} \pi |B|}{cn}+O(e^{-cn}).\label{eq:cont3}
	 \end{align}
Hence, we obtain by adding the contributions from \eqref{MRdens}, \eqref{eq:cont2} and \eqref{eq:cont3} that for all $z \in B$, 
\begin{align}
	\E [(\xi_n \Delta \xi_n^{z,\varnothing}) (\sqrt nS_{n,z})] \le 2|B|\Big( \frac{ \pi}{n^2} +  \frac{2 \sqrt{\log n} + 3^{5/2} \pi }{cn} \Big) +O(e^{-cn}). \label{eq:xisymmbou}
\end{align}
This completes the bound of the expectation of the right-hand side of \eqref{eq:symm1}.

Finally, we consider the sum in \eqref{eq:symm2}. For all $z \in B$ it holds almost surely that
	\begin{align*}
		&\sum_{w \in \xi_n^{z,\varnothing} \cap \xi_n \cap (\sqrt n B\setminus \sqrt nT_{n,z})} \big|\one \{\xi_n^{z,\varnothing}(B_{r_n}(w))=0\}-\one \{\xi_n(B_{r_n}(w))=0\} \big|\\
		&\quad \le \sum_{w \in \xi_n^{z,\varnothing} \cap \xi_n \cap (\sqrt n B\setminus \sqrt n T_{n,z})} (\xi_n^{z,\varnothing} \Delta \xi_n) (B_{r_n}(w)) 
\max\big(\one\{\xi_n^{z,\varnothing}(B_{r_n}(w)=0)\}, \one\{\xi_n(B_{r_n}(w)=0)\}\big).
	\end{align*}
By Lemma \ref{lem:int} there exists $\delta>0$ such that $B_{r_n}(w)\subset B_{(1-\delta)\sqrt n}(0)$ for all $w \in B$ and $n\ge C(s, \kappa)$. Hence, the previous sum is bounded by
\begin{align}
\begin{split}
\label{e:bound29}
		&\sum_{v \in \xi_n^{z,\varnothing} \Delta \xi_n}  \one\{v \in B_{(1-\delta)\sqrt n}(0) \setminus \sqrt n T_{n,z}\} \sum_{w \in \xi_n^{z,\varnothing} \cap \xi_n \cap B_{r_n(w)}(v)} \max\{\one\{\xi_n^{z,\varnothing}(B_{r_n}(w))=0\}, \one\{\xi_n(B_{r_n}(w))=0\}\} \\
		&\quad \le \sum_{v \in \xi_n^{z,\varnothing} \setminus \xi_n} \one\{v \in B_{(1-\delta)\sqrt n}(0) \setminus \sqrt n T_{n,z}\}\sum_{w \in \xi_n \cap B_{r_n(w)}(v)} \one\{\xi_n(B_{r_n}(w))=0\} \\
		&\quad \quad +\sum_{v \in \xi_n \setminus \xi_n^{z,\varnothing}} \one\{v \in B_{(1-\delta)\sqrt n}(0) \setminus \sqrt n T_{n,z}\}\sum_{w \in \xi_n^{z,\varnothing} \cap B_{r_n(w)}(v)} \one\{\xi_n^{z,\varnothing}(B_{r_n}(w))=0\}.
\end{split}
\end{align}
Using 
	\[\lim_{n \rightarrow \infty} \sup_{v,w \in B_{(1-\delta)\sqrt n}(0)} \frac{r_n(v)}{r_n(w)}  = 1,\] from Lemma \ref{lem:int}, we find by a straightforward volume argument (considering how many balls with radius $r_n(w)/2$ fit into a ball with radius $2 r_n(v)$) that 
	\[ \sum_{w\in \omega \cap B_{r_n}(v)} \one\{\omega(B_{r_n}(w))=0\} \le 16\] for all locally finite $\omega \in \mathbf N$ and $n$ large enough. Inserting this estimate into \eqref{e:bound29} gives that \eqref{eq:symm2} is almost surely bounded by 
	$$
 	16 (\xi_n^{z,\varnothing} \Delta \xi_n) (B_{(1-\delta)\sqrt n}(0)\setminus \sqrt nT_{n,z}).
	$$
	Hence, by \eqref{eq:xisymmbou} (with $B$ replaced by $B_{(1-\delta)\sqrt n}(0)$), we have for all $z\in B$ that \eqref{eq:symm2} is bounded by
	$$
	 32|B_{1-\delta}(0)| \Big( \frac{ \pi}{n^2} +  \frac{2 \sqrt{\log n} + 3^{5/2} \pi }{cn} \Big) +O(e^{-cn}).
	$$
This completes the bound for \eqref{eq:symm2}, and hence completes the proof. 
\end{proof}

\begin{proof}[Proof of Corollary~\ref{cor:maxdistance}]
Fix $\epsilon >0$. We first show that \[\P\left(\max\limits_{z \in \xi_n \cap \sqrt nB} \min\limits_{w \in \xi_n\setminus \{z\}} |z-w| \ge (\sqrt{2}+\epsilon)\sqrt[4]{\log n}\right) \to 0\] as $n \to \infty$. Using the definition of $\Xi_{n}$, that $L_n$ is the intensity measure for $\Xi_{n}$, and $\rho_n(z) \le \pi^{-1}$ for all $z \in B$, we have
\begin{align}
\P\Big(\max_{z \in \xi_n \cap \sqrt nB}  \min\limits_{w \in \xi_n\setminus \{z\}} |z-w| \ge (\sqrt{2}+\epsilon)\sqrt[4]{\log n}\Big) &\le \mathbb E \sum_{z \in \xi_n \cap \sqrt nB} \one\Big\{ \min\limits_{w \in \xi_n\setminus \{z\}} |z-w| \ge (\sqrt{2}+\epsilon)  \sqrt[4]{\log n} \Big\} \notag \\
&= \int_{\sqrt n B} \P(\xi_n^{z!}(B_{(\sqrt{2}+\epsilon)\sqrt[4]{ \log n}}(z))=0) \rho_n(z)\,  \d z\notag \\
&\le  
\frac{n |B|}{\pi} \sup_{z \in \sqrt n B} \P(\xi_n(B_{(\sqrt{2}+\epsilon)\sqrt[4]{ \log n}}(z))\le 1),\label{e:cor2tozero}
\end{align}
where in the last inequality we used the coupling of of $\xi_n^{z!}$ and $\xi_n$ indicated in Section~\ref{sec:coup}.
By Lemma \ref{lem:mom2} with $s_n = (\sqrt{2}+\epsilon)\sqrt[4]{\log n}$ we have  for $n \ge C(\epsilon, B)$ that 
\[
\sup_{z \in \sqrt n B} \P(\xi_n(B_{(\sqrt{2}+\epsilon)\sqrt[4]{\log n}}(z))\le 1) \le n^{-(\sqrt{2}+\epsilon)^4 /4 }.\] Hence, \eqref{e:cor2tozero} tends to $0$ as $n \to \infty$.

Next, we show that \[\P\Big(\max\limits_{z \in \xi_n \cap \sqrt nB}  \min\limits_{w \in \xi_n\setminus \{z\}} |z-w| \le (\sqrt{2}-\epsilon) \sqrt[4]{\log n}\Big) \to 0.\]
 Fix $\kappa>0$. By \eqref{e:rcnlimit}, we have $(\sqrt{2}-\epsilon)  \sqrt[4]{\log n} \le r_{n}(z)$ for $n\ge C(\epsilon, \kappa, B)$.  
  Therefore,
\begin{align*}
	\P\Big(\max_{z \in \xi_n \cap \sqrt nB} \min_{w \in \xi_n\setminus \{z\}} |z-w| \le (\sqrt{2}-\varepsilon) \sqrt[4]{\log n} \Big) &\le \P\Big(\max_{z \in \xi_n \cap \sqrt nB} \min_{w \in \xi_n\setminus \{z\}} |z-w| \le  r_{n}(z) \Big)\\
	&=\P(\Xi_{n} \cap B =\varnothing),
\end{align*}
which tends to $\P(\zeta_{\kappa} \cap B= \varnothing) = \exp(-\kappa |B|)$ as $n \to \infty$ by Theorem \ref{th}, since convergence in KR distance implies convergence in distribution for locally finite point processes (see \cite[Proposition 2.1]{DST16}). 
Since $\kappa$ was arbitrary, this upper bound can be made arbitrarily small by taking $\kappa$ sufficiently large. 
\end{proof}
\noindent
{\bf Acknowledgments.} 
PL was partially supported by NSF grant DMS-2450004.
MO was supported by the NWO Gravitation project NETWORKS under grant agreement no.~024.002.003. The authors thank the anonymous referees for their many helpful comments.

\bibliographystyle{amsplain}

\bibliography{ngin}

\end{document}